\documentclass[a4paper,reqno]{amsart}
\usepackage{snapshot}
\usepackage{amsthm, amssymb, amsmath, latexsym}
\usepackage{amsfonts}
\usepackage{graphicx}
\graphicspath{{./figures/}}
\usepackage{color}
\usepackage{pdfsync}

\usepackage{my-plots}

\theoremstyle{plain}
\begingroup
\newtheorem{theorem}{Theorem}[section] 
\newtheorem{lemma}[theorem]{Lemma}

\endgroup

\theoremstyle{definition}
\begingroup

\newtheorem{remark}[theorem]{Remark}

\endgroup

\theoremstyle{remark}
\begingroup
\endgroup

\mathchardef\emptyset="001F

\numberwithin{equation}{section}

\def \Om {\Omega}
\def \R {\mathbb R}

\def \d {\delta}

\newcommand\ee{\end{equation}}
\newcommand\be{\begin{equation}}
\renewcommand{\d}{\ensuremath{\,\mathrm{d}}}

\title[Jerky crack growth in elastoplastic materials]
{A numerical study of the jerky crack growth in elastoplastic materials with localized plasticity}

\author[Gianni Dal Maso]{Gianni Dal Maso}
\address[Gianni Dal Maso]{SISSA, Via Bonomea 265, 34136 Trieste,
Italy}
\email[Gianni Dal Maso]{dalmaso@sissa.it}

\author[Luca Heltai]{Luca Heltai}
\address[Luca Heltai]{SISSA, Via Bonomea 265, 34136 Trieste,
Italy}
\email[Luca Heltai]{luca.heltai@sissa.itt}

\begin{document}

\thanks{Preprint SISSA 06/2020/MATE}

\begin{abstract}

We present a numerical implementation of a model of quasi-static crack growth in linearly elastic-perfectly plastic materials. We assume that the displacement is antiplane, and that the cracks and the plastic slips are localized on a prescribed path. We provide numerical evidence of the fact that the crack growth is intermittent, with jump characteristics that depend on the material properties.

\end{abstract}

\maketitle



\begin{center}
    \it Dedicated to Umberto Mosco on the occasion of his 80th birthday
\end{center}
\section{Introduction}\label{intro}
Models of quasi-static crack growth in elasto-plastic materials provide a description of fracture which is more accurate with respect to the classic purely brittle approach (see, for example,~\cite{Rice,Rice-Sor,Hutch} and, for a numerical treatment, \cite{Ortiz1999, DeLorenzisMcBrideReddy2016, WickWickHellmig2015}).  Indeed,  in real materials the crack tip is always surrounded by a process zone where several nonlinear phenomena occur, among which plastic slip is one of the most relevant.

The case of crack growth in linearly elastic--perfectly plastic materials in the small strain regime was considered in~\cite{DM-Toa} in the framework of the energetic approach to rate independent systems (for which we refer to~\cite{Mie-Rou} and the references therein).

Recent developments~\cite{DM-Toa2} have shown that in this model the crack growth is jerky. To be precise, the mathematical proof of this result has been obtained in a simplified setting in the antiplane case in dimension two, where both the plasticity and the crack are confined on a prescribed path. This jerky behavior is in agreement with experimental results~\cite{Hull,Descatha} and with numerical evidence (see, for example,~\cite{Bra-Tan-Bour-Bha, Ortiz1999}).

In this work we provide a numerical implementation of the model introduced in~\cite{DM-Toa2}, under suitable  symmetry assumptions on the domain and on the boundary conditions. We confirm numerically that the crack growth is indeed intermittent, at least in those regimes where the plastic behavior cannot be neglected. Moreover, we provide examples which show the dependency of the number of jumps on the problem parameters. 

Thanks to some symmetry assumptions, the contribution of the plastic part of the problem reduces to a non-homogenous Neumann boundary condition on the cohesive portion of the crack path, which is itself an unknown of the problem (see Figure~\ref{BVP}).

When considering a rectangular domain and prescribing a straight crack path, we prove that the cohesive zone is a segment, and we provide a constructive way to identify the coordinates of the end points of this segment. This is done by solving a series of standard Laplace equations in the rest of the domain in order to minimize a suitable energy with respect to the coordinates of the end points. The presented strategy allows us to avoid the typical difficulties that are present in the solution of classical minimization problems of elasto-plasticity.

\section{Formulation of the problem}\label{Setting}
 
We consider a simplified quasi-static model of crack growth in elasto-plastic materials, in the discrete time formulation with prescribed crack path. Our model is antiplane, with reference configuration given by  the rectangle
\begin{equation}\label{ex1111}
\Om:=(0,a)\times(-b,b),
\end{equation}
for some $a>0$ and $b>0$. 

Given $T$ and $n$, let $t_i:=iT/n,$, $i=0,\ldots,n$ be a discrete sampling of the time interval $[0,T]$. The displacement at time $t_i$ is given by a scalar function $u_i\colon\Omega\to \R$, which takes a prescribed value on the Dirichlet part of the boundary, given by 
\begin{equation}\label{dirichlet boundary conditions}
\partial_D\Omega:=[0,a]\times\{-b,b\}.
\end{equation}
Namely, we assume that at time $t_i$
\begin{equation}\label{boundary condition}
u_i(x,\pm b)=\pm w_i(x)\quad \hbox{for every } x \in (0,a),
\end{equation}
where the functions $w_i \colon [0,a] \mapsto \R$ satisfy the following three conditions:
\begin{eqnarray}
&w_i\in C^\infty([0,a]) \hbox{ and not identically zero for } i=0,1,\ldots,n,\label{w(t) Cinfty}
\\
&0\leq w_{i-1}\leq w_i \quad\hbox{on }[0,a]\hbox{ for } i=1,\ldots,n,
\label{increasing in t}
\\
&w_i\hbox{ is nonincreasing and nonnegative on }[0,a]\hbox{ for } i=0,1,\ldots,n.
\label{decreasing in x}
\end{eqnarray}

The crack path is the segment
\begin{equation}\label{gamma}
\Gamma:=[0,a]\times\{0\},
\end{equation}
and in our model we assume also that the plastic strain is concentrated on $\Gamma$. For convenience, we indicate the set $[s_1, s_2]\times\{0\}$ with $\Gamma_{s_1}^{s_2}$.

The crack at time $t_i$ has the form $\Gamma_0^{s_i}$ for some $s_i\in[0,a]$, which represents the position of the crack tip. 
Since the plastic strain is concentrated on $\Gamma$, for every $t_i$ the corresponding displacement $u_i$ belongs to $H^1(\Om\setminus\Gamma)$. The plastic slip is localized
on $\Gamma_{s_i}^a$ and is given by $[u_i]$, where, for every $v\in
H^1(\Om\setminus\Gamma)$, we set  $[v](x) = \lim_{y\to
  0^+}\big(v(x,y)-v(x,-y)\big)$ in the sense of traces. We refer to
Figure~\ref{fig:domain-scheme} for a schematic of the domain and of
the boundary conditions.
\begin{figure}
  \centering
  \includegraphics{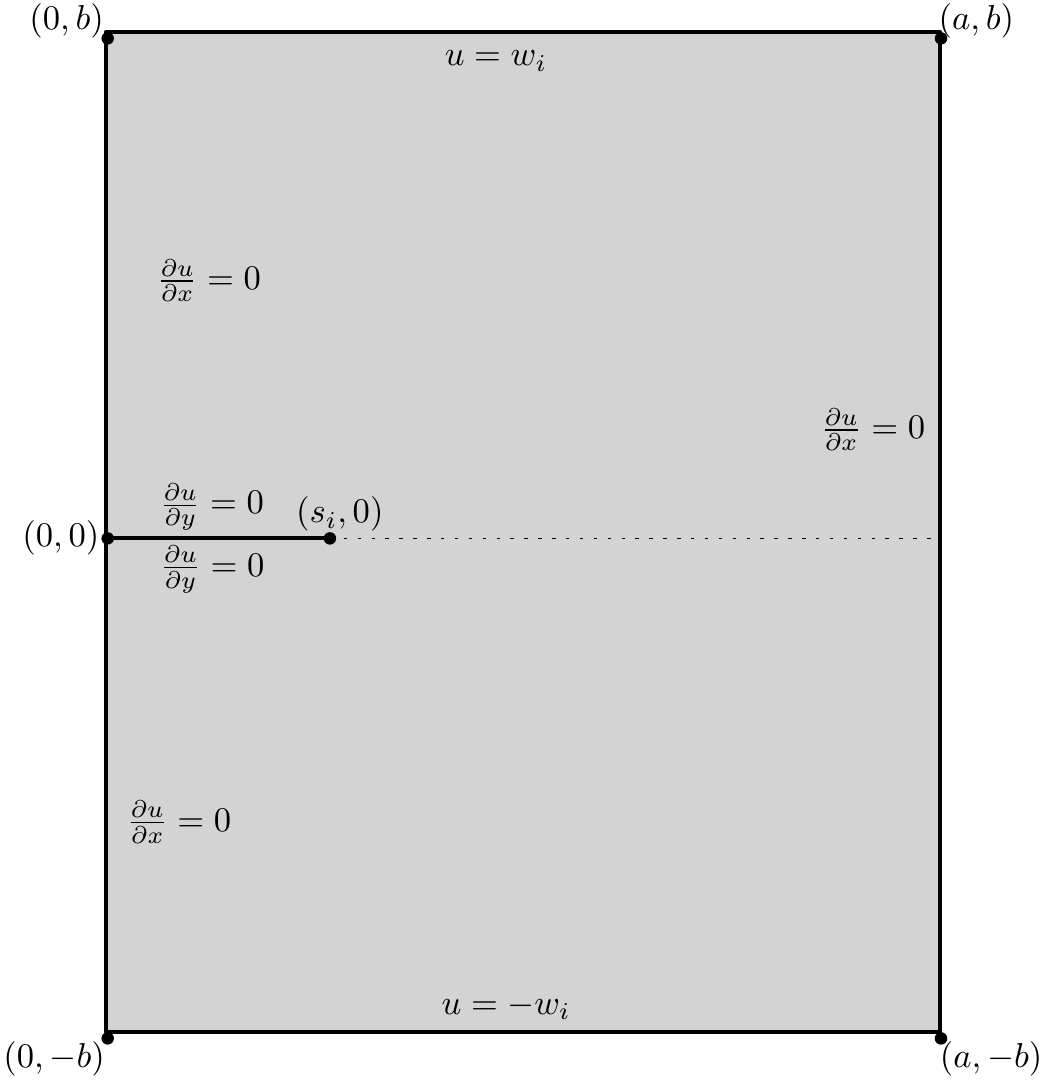}
  \caption{Schematic of the domain, boundary conditions, and crack tip
    position.}
  \label{fig:domain-scheme}
\end{figure}

We consider a model where the stored elastic energy at time $t_i$ is given by 
\begin{equation}\label{bulk-energy}
\mathcal E^e_i := \frac \alpha 2 \int_{\Om\setminus\Gamma}|\nabla
u_i|^2\d x\d y,
\end{equation}
the plastic dissipation distance between  $u_{i-1}$ and  $u_{i}$ is given by
\begin{equation}\label{plastic-energy}
\mathcal E^p_i:= \beta \int_{\Gamma_{s_i}^a}|[u_i]-[u_{i-1}]| \d x,
\end{equation}
while the energy dissipated by the crack growth in the time interval between $t_{i-1}$ and $t_i$ is given by
\begin{equation}
\mathcal E^c_i := \gamma (s_i-s_{i-1}).
\end{equation}
Here and henceforth $\alpha$, $\beta$, and $\gamma$ are positive material constants. 

We assume that the initial condition  $u_0$ for the displacement belongs to $H^1(\Om)$ and satisfies
\begin{equation}\label{initial u0}
\begin{cases}
\Delta u_0=0 & \text{in }\Om\,,\\
u_0(\cdot,\pm b)=\pm w_0 & \hbox{on } (0,a)\,,\\
\frac{\partial u_0}{\partial x}(0,\cdot)=\frac{\partial u_0}{\partial x}(a,\cdot)=0 & \text{on }(-b,b)\,.
\end{cases}
\end{equation}
Let  $s_0\in(0,a)$ be the initial condition for the crack tip. In our model, the pair $(u_i, s_i)$, with $i=1,\ldots,n$, is obtained inductively as a solution of the minimum problem 
\begin{equation}\label{minii}
\min_{\substack{u\in H^1(\Om\setminus\Gamma)\\ u=w_i\ \text{on }\partial_D\Om\\s_{i-1}\le s\le a
}}\Big\{ \frac \alpha 2 \int_{\Om\setminus\Gamma}|\nabla u|^2\d x\d y+\beta \int_{\Gamma_{s}^a}|[u]-[u_{i-1}]|\d x+\gamma(s-s_{i-1})\Big\}.
\end{equation}
In general this solution is not unique, but two solutions with the same value of $s$ must coincide by strict convexity in $u$. By considering the Euler equation of \eqref{minii} we see that the second term produces a cohesive force equal to $\pm\beta$ at the points of $\Gamma_{s}^a$ where $\pm([u]-[u_{i-1}])<0$.

Due to the symmetry of the data we shall work in  $\Om^+:=(0,a)\times(0,b)$ with Dirichlet boundary conditions on $\partial^+_D\Om:=[0,a]\times\{b\}$. This is justified by the following remark.  
\begin{remark}\label{odd}
Since the boundary condition \eqref{boundary condition} is odd with respect to $y$, a pair
 $(u_i,s_i)$ is a solution of \eqref{minii} if and only if $u_i$ is odd with respect to $y$ and $(u_i|_{\Om^+},s_i)$ is a solution of the minimum problem 
\begin{equation}\label{minii*}
\min_{\substack{u\in H^1(\Om^+)\\ u=w_i\ \text{on }\partial^+_D\Om\\s_{i-1}\le s\le a
}}\Big\{ \frac \alpha 2 \int_{\Om^+}|\nabla u|^2\d x\d y+\beta \int_{\Gamma_{s}^a}|u-u_{i-1}|\d x+\frac{\gamma}{2}(s-s_{i-1})\Big\}\,.
\end{equation}
Moreover, $u_0$ is odd  with respect to $y$ and $u_0|_{\Om^+}$ is the solution of the boundary value problem
\begin{equation}\label{initial u0*}
\begin{cases}
\Delta u=0 & \text{in }\Om^+\,,\\
u= w_0 & \hbox{if } y=b\,,\\
\frac{\partial u}{\partial x}=0 & \text{if }x=0\text{ or }x=a\,,\\
u=0 & \hbox{if } y=b\,.
\end{cases}
\end{equation}
\end{remark}
To solve \eqref{minii*}, for every $s_{i-1}\le s\le a$ we find the unique solution $u^s_i$ of the problem
\begin{equation}\label{min given s}
\min_{\substack{u\in H^1(\Om^+)\\ u=w_i\ \text{on }\partial^+_D\Om
}}\Big\{ \frac \alpha2 \int_{\Om^+}|\nabla u|^2\d x\d y+\beta\int_{\Gamma_{s}^a}|u-u_{i-1}|\d x\Big\}\,,
\end{equation}
and then we find a value $s_i$ of $s$ which solves the minimum problem
\begin{equation}\label{min in s}
\min_{s_{i-1}\le s\le a}
\Big\{ \frac \alpha2 \int_{\Om^+}|\nabla u^s_i|^2\d x\d y+\beta\int_{\Gamma_{s}^a}|u^s_i-u_{i-1}|\d x +\frac\gamma2(s-s_{i-1})\Big\}\,.
\end{equation}
The pair $(u_i,s_i)$, with $u_i:=u^{s_i}_i$, is a solution of \eqref{minii*}.

The solution of \eqref{min given s} will be constructed by using the solutions of some auxiliary boundary value problems for the Laplace equation, depending on a parameter $\sigma$, 
and then by choosing a particular value $\sigma^s_i$ of this parameter. 
Let us fix $ s\in[s_0,a]$. For every $\sigma\in[s,a]$ we consider the solution $u_i^{s,\sigma}\in H^1(\Om^+)$ of the problem (see Fig.\ \ref{BVP})
\begin{equation}\label{eqst}
\begin{cases}
\Delta u=0 & \text{in }\Om^+,\\
u=w_i &\text{if }y=b\,,\\
\frac{\partial u}{\partial x}=0 & \text{if }x=0\text{ or }x=a\,,\\
\frac{\partial u}{\partial y}=0 &\text{for }0<x<s \text{ and }y=0\,,\\
\frac{\partial u}{\partial y}=\frac{\beta}{\alpha} &\text{for }s<x<\sigma \text{ and }y=0\,,\\
u=0 &\text{for }\sigma<x<a \text{ and }y=0\,.
\end{cases}
\end{equation}
By the continuous dependence on the data, the function $u_i^{s,\sigma}$ is continuous in $H^1(\Om^+)$ 
with respect to $\sigma$.

\begin{figure}[ht!]
\includegraphics{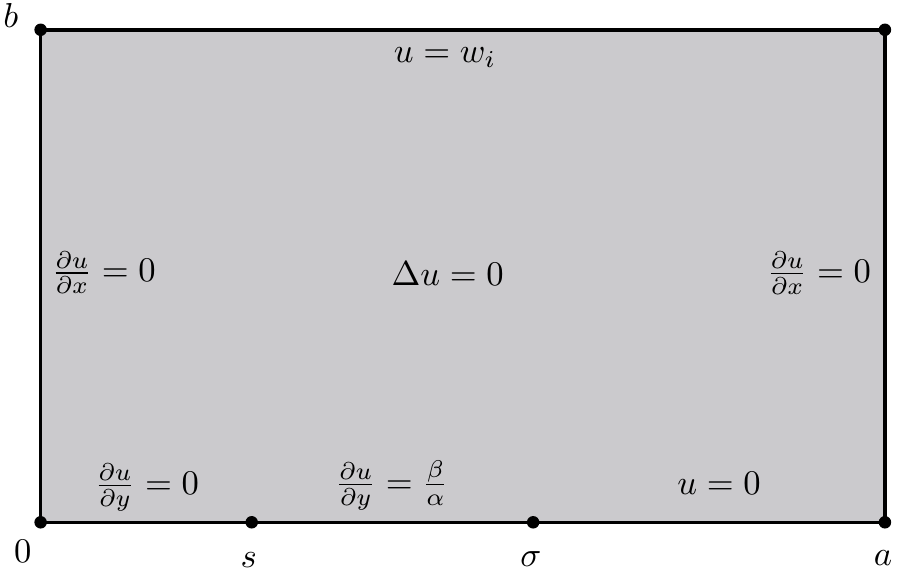}
\caption{The boundary value problem for $u_i^{s,\sigma}$.}
 \label{BVP}
\end{figure}

We define 
\begin{equation}\label{080703}
\sigma^s_i:=\max\{\sigma\in[s,a]:u_i^{s,\sigma}\ge0 \text{ in }\Om^+\}.
\end{equation}
The existence of the maximum follows easily from the continuous dependence of $u^{s,\sigma}$ on $\sigma$. We shall prove that the solution $u_i^s$ of \eqref{min given s} is given by
\begin{equation}
u_i^s=u_i^{s,\sigma}\,,\text{ with }\sigma=\sigma^s_i.
\end{equation}

\begin{remark}\label{inequality on s sigma}
By the maximum principle the inequality $u_i^{s,\sigma}\ge0$ in $\Om^+$ is satisfied if and only if the trace of $u_i^{s,\sigma}$ satisfies $u_i^{s,\sigma}(x,0)\ge0$ for every $x\in(s,\sigma)$. Therefore we have
\begin{equation}\label{08070378}
\sigma^s_i:=\max\{\sigma\in[s,a]:u_i^{s,\sigma}(x,0)\ge0 \text{ for every }x\in(s,\sigma)\}\,.
\end{equation}
\end{remark}

The following lemmas shows the connection between the minimum problem \eqref{min given s} and the functions 
$u_i^{s,\sigma}$. 

\begin{lemma}\label{main lemma}
 Let  $s\in[s_0,a)$, $i = 1,\ldots,n$, let $u_i^{s,\sigma}$ be defined as in \eqref{eqst}, with $\sigma=\sigma^s_i$ defined as in \eqref{080703}, and let $z\in H^1(\Om^+)$ be
 such that $0\le z\le u_i^{s,\sigma}$ a.e.\ on $\Gamma_{s}^a$. 
Then $ u_i^{s,\sigma}$ is the solution of the minimum problem
\begin{equation}\label{min given s and w0}
\min_{\substack{u\in H^1(\Om^+)\\ u=w_i\ \text{on }\partial^+_D\Om
}}\Big\{ \frac \alpha2 \int_{\Om^+}|\nabla u|^2\d x\d y+\beta\int_{\Gamma_s^a}|u-z|\d x\Big\}\,.
\end{equation}
\end{lemma}

\begin{proof} The definition of $u_i^{s,\sigma}$ given here differs from the definition in \cite[formula (5.15)]{DM-Toa2} only in the boundary condition at $y=b$, which is now the nondecreasing function $w_i$. The proof in our case can be obtained by repeating the proofs of Lemmas 5.7, 5.8, 5.9, and 5.11 of \cite{DM-Toa2}, with minor changes. In particular, the assumption that $w_i$ is nonincreasing is used in the proof of Lemma 5.7, where the $C^\infty$-regularity of $u_i^{s,\sigma}$ at the vertices $(0,b)$ and
$(a,b)$ cannot be obtained for a nonconstant $w_i$, with no consequences on the conclusion.
\end{proof}

We now want to prove that the solution $u^s_i $ of problem \eqref{min given s} is given by $u^s_i = u_i^{s,\sigma}$, with $\sigma=\sigma^s_i$, where  $u_i^{s,\sigma}$ is defined by \eqref{eqst} and $\sigma^s_i$ by \eqref{080703}.
This will be a consequence of  Lemma \ref{main lemma} if we show that $u_{i-1}\le u_i^{s,\sigma}$ for $s\ge s_{i-1}$ and $\sigma=\sigma^s_i$. The tools to prove this inequality are given by the next lemma. For every $x$ and $y$ in $\R$,  we set $x\lor y: = \max\{x,y\}$ and $x \land y := \min\{x,y\}$.

\begin{lemma}\label{main lemma*}
Let $i,\,j\in\{1,\ldots,n\}$, with  $i\leq j$, and let  $s\in [s_{i-1},a]$, where $s_{i-1}$ is defined in \eqref{minii*} for $i \geq 2$. Then, setting $\sigma_{i-1}:=\sigma^{s_{i-1}}_{i-1}$ and $\sigma_j:=\sigma^{s}_j$ according to \eqref{080703},  we have
\begin{align}
 &u_{i-1}^{s_{i-1}, \sigma_{i-1}}\le u_j^{s, \sigma_j} \hbox{ in }\Omega^+\,,\label{monotonicity of u}\\
 & \sigma_{i-1} \leq  \sigma_j\,. \label{monotonicity of sigma}
\end{align}
\end{lemma}
\begin{proof} By the maximum principle we have
\begin{equation}\label{u hat u}
0\le u_{i-1}^{s_{i-1}, \sigma_{i-1}}\le  u_j^{s_{i-1}, \sigma_{i-1}}\quad\hbox{ in }\Omega^+,
\end{equation}
where we compare two solutions of problem \eqref{eqst}, with $s=s_{i-1}$, and $\sigma=\sigma_{i-1}$, changing only the boundary condition from  $w_{i-1}$ to $w_j$, which satisfy $w_{i-1}\leq w_j$ by \eqref{increasing in t}.
To prove \eqref{monotonicity of u} it is enough to show that
\begin{equation}\label{hat u s hat u hat s}
 u_j^{s_{i-1}, \sigma_{i-1}} \leq u_j^{s, \sigma_j} \hbox{ in }\Omega^+\,.
\end{equation}
Let $v:=u_j^{s, \sigma_j} - u_j^{s_{i-1}, \sigma_{i-1}} $. Then 
$v\in H^1(\Om^+)$  and $\Delta v=0$ in $\Om^+$. See Figure~\ref{auxiliary bvp on v} for a schematic of the boundary conditions satisfied by $v$, that will be explained below.

\begin{figure}
\includegraphics{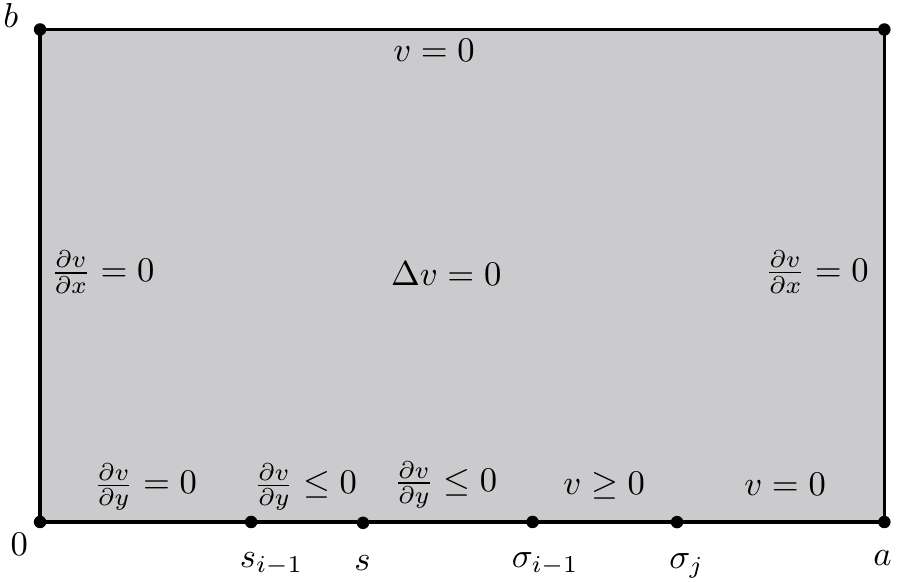}
\caption{Boundary condition for $v$ when $s<\sigma_{i-1}$.}
\label{auxiliary bvp on v}
\end{figure}


Since
$u_j^{s, \sigma_j} =u_j^{s_{i-1}, \sigma_{i-1}} = w_j$ for $y=b$, we have $v=0$ for $y=b$.
The equalities $\frac{\partial}{\partial x}u_j^{s, \sigma_j}  = \frac{\partial}{\partial x}u_j^{s_{i-1}, \sigma_{i-1}}=0$
for  $x=0$ or $x=a$ give 
\begin{equation}\label{zero derivarive}
\frac{ \partial v}{\partial x}=0~\hbox{ for }~x=0\hbox{ or }x=a\,.
\end{equation}

Since $u_j^{s_{i-1}, \sigma_{i-1}} \ge 0$ in $\Omega^+$  by \eqref{u hat u}  and  $u_j^{s_{i-1}, \sigma_{i-1}} = 0$  for $\sigma_{i-1}<x<a$ and $y=0$ by \eqref{eqst}, we have $\frac{\partial}{\partial y}u_j^{s_{i-1},\sigma_{i-1}} \ge 0$ for $\sigma_{i-1}<x<a$ and $y=0$. By \eqref{eqst} we have also
$\frac{\partial}{\partial y}u_j^{s_{i-1}, \sigma_{i-1}}=0$ for $0<x<s_{i-1}$ and $y=0$ and $\frac{\partial}{\partial y}u_j^{s_{i-1}, \sigma_{i-1}}=\frac{\beta}{\alpha}$  for $s_{i-1}<x<\sigma_{i-1}$ and $y=0$, hence
$\frac{\partial}{\partial y}u_j^{s_{i-1}, \sigma_{i-1}}\ge 0$ for $x\in(0,a)\setminus\{s_{i-1},\sigma_{i-1}\}$ and $y=0$.
Since $\frac{\partial}{\partial y}u_j^{s, \sigma_j}=0$ for $0<x<s$ and $y=0$  by \eqref{eqst}, 
the previous inequality gives
\begin{equation}\label{negative derivative}
\frac{\partial v}{\partial y}\le 0\hbox{ for }x\in(0, s)\setminus\{s_{i-1},\sigma_{i-1}\}\hbox{ and }y=0\,.
\end{equation}

By \cite[Lemma 5.9]{DM-Toa2}, which is still valid with our definition of $u_j^{s, \sigma_j}$, we have $\frac{\partial}{\partial y}u_j^{s, \sigma_j} \le \frac{\beta}{\alpha}$ for $\sigma_j<x<a$ and $y=0$. Since $\frac{\partial}{\partial y}u_j^{s, \sigma_j}= \frac{\beta}{\alpha}$ for $ s<x<\sigma^s_j$ and $y=0$ by \eqref{eqst}, we have $\frac{\partial}{\partial y}u_j^{s, \sigma_j}\le \frac{\beta}{\alpha}$ for $x\in (s,a)\setminus \{\sigma_j\}$ and $y=0$. 
Using the equality $\frac{\partial}{\partial y} u_j^{s_{i-1}, \sigma_{i-1}} =\frac{\beta}{\alpha}$  for $s_{i-1}<x<\sigma_{i-1}$ and $y=0$, given  by \eqref{eqst}, we deduce that
\begin{equation}\label{negative derivative 2}
\frac{\partial v}{\partial y}\le 0\hbox{ for }x\in (s,s \lor \sigma_{i-1})\setminus \{\sigma_j\} \hbox{ and }y=0\,.
\end{equation}

Since $u_j^{s, \sigma_j}\ge 0$ for $x\in(0,a)\setminus\{s,\sigma_j\}$ and $y=0$ by \eqref{080703}, while $u_j^{s_{i-1}, \sigma_{i-1}} =0$ for $ \sigma_{i-1}<x<a$ and $y=0$  by \eqref{eqst}, we have
\begin{equation}\label{v positive}
v\ge 0\hbox{ for }x\in    (s \lor \sigma_{i-1}, a)\setminus\{\sigma_j\} \hbox{ and }y=0\,.
\end{equation}

We now prove that, using a suitable version of the maximum principle, the boundary conditions discussed above (see Figure \ref{auxiliary bvp on v}) imply that $v\geq 0$ in $\Omega^+$.
Let $\varphi\in H^1(\Om^+)$ with $\varphi=0$ for $y=b$ and for $s \lor \sigma_{i-1}<x<a$ and $y=0$. Suppose in addition that $\varphi$ vanishes in a neighborhood of the vertices of the rectangle $\Om^+$ and of the points $(s_{i-1},0)$, $(\sigma_{i-1},0)$, $(s,0)$, $(\sigma_j,0)$.

Since $\int_{\Omega^+}\Delta v \varphi\,\d x\d y=0$, integrating by parts from \eqref{zero derivarive} we obtain the weak formulation 
\begin{equation}\label{weak formulation}
\int_{\Omega^+}\nabla v\nabla\varphi\,\d x\d y=-\int_{\Gamma_0^{s \lor \sigma_{i-1}}}\frac{\partial v}{\partial y}\varphi\, \d x
\end{equation}
Since points have capacity zero, every function in $H^1(\Omega^+)$ can be approximated in $H^1(\Omega^+)$ by functions that vanish in a neighborhood of a prescribed finite set of points. Therefore \eqref{weak formulation} holds for every $\varphi\in H^1(\Omega^+)$ with $\varphi=0$ for $y=b$ and for $s \lor \sigma_{i-1}<x<a$ and $y=0$. 
Taking $\varphi:=v\land 0$ we obtain
$$
\int_{\Omega^+}|\nabla (v\land 0)|^2=\int_{\Omega^+}\nabla v\nabla(v\land 0)\,\d x\d y=-\int_{\Gamma_0^{s \lor \sigma_{i-1}}}\frac{ \partial v}{\partial y}(v\land 0)\d x\,.
$$
Since $\frac{ \partial v}{\partial y}\le 0$ a.e.\ in $\Gamma_0^{s \lor \sigma_{i-1}}$ by \eqref{zero derivarive}-\eqref{negative derivative 2}, the right-hand side of the previous equality is less than or equal to $0$,
which gives $\nabla(v\land 0)=0$ in $\Omega^+$. Taking into account the boundary condition $v=0$ on $\partial_D^+\Omega$
we get $v\land 0=0$ in $\Omega^+$. This implies $v\ge0$, which proves~\eqref{hat u s hat u hat s}.

Property \eqref{monotonicity of sigma} is trivial if $\sigma_{i-1}\leq s$, since $s\leq\sigma_j$ by \eqref{080703}. If $s<\sigma_{i-1}$  inequality \eqref{monotonicity of sigma} follows from \eqref{080703} and from the inequality $u_{i-1}^{s_{i-1},\sigma_{i-1}}\leq u^{s,\sigma_{i-1}}_j$, which can be proved by adapting the argument used in the proof of the inequality $v\geq 0$.
\end{proof}

We are now able to prove the result which allows us to construct the solution of the minimum problems \eqref{min given s} by means of the solutions of \eqref{eqst}.

\begin{lemma}\label{main lemma**}
Assume that \eqref{w(t) Cinfty}-\eqref{decreasing in x} hold. Then for every $i=1,\dots,n$ and every $s_{i-1}\le s\le a$ the solution $u^s_i$ of the minimum problem \eqref{min given s} coincides with the function $u^{s,\sigma}_i$ defined in \eqref{eqst}, with boundary data $w_i$ and with $\sigma=\sigma^s_i$ defined in~\eqref{080703}. 
\end{lemma}

\begin{proof}
The proof is obtained by induction on $i$ using Lemmas \ref{main lemma} and \ref{main lemma*}. Since $u_0$ satisfies \eqref{initial u0*}, for every  $s_0\leq s\leq a$ by \eqref{080703} we have $u_0=0\leq u^{s,\sigma}_1$ for $y=0$ and $\sigma=\sigma^s_1$, while  $0 \leq u_0 = w_0\le w_1=u^{s,\sigma}_1$ for  $y=0$ and  $\sigma=\sigma^s_1$ by \eqref{increasing in t}. Using the comparison principle we obtain
$0\leq u_0\leq u_1^{s,\sigma}$in $\Omega^+$ for $s_0\leq s\leq a$ and  $\sigma=\sigma^s_1$. Then $u^s_1=u^{s,\sigma}_1$ with  $\sigma=\sigma^s_1$ by Lemma \ref{main lemma} with $z=u_0$.

Assume now that the statement of the Lemma is true for step $i-1$. Then in particular $u_{i-1}=u^{s_{i-1}}_{i-1}=u^{s_{i-1}, \sigma}_{i-1}$ with $ \sigma=\sigma^{s_{i-1}}_{i-1}$. The statement of the Lemma for step $i$ follows from Lemma \ref{main lemma}, applied with $z=u_{i-1}$. We see that the hypothesis of Lemma \ref{main lemma}  are satisfied using Lemma \ref{main lemma*} with  $j = i$.
\end{proof}

\section{Finite element approximation}\label{FEM}

In order to provide a numerical approximation to problem
\eqref{minii*}, we  fix a family  $\{\mathcal T_h(\Omega^+)\}_{h>0}$ of conforming subdivisions of $\overline\Omega^+$ 
made of
quadrilaterals, with $h$ denoting their maximal size. We assume that
$\mathcal T_h(\Omega^+)$ are shape-regular and quasi-uniform in the
sense of \cite{Ciarlet2002,ErnGuermond2004}, and that they contain the point $(s_0,0)$ among their vertices.  Each subdivision $\mathcal T_h(\Omega^+)$ induces a discrete set $\mathcal S_h$ in the segment $[s_0, a]\times\{0\}$, composed of all vertices of $\mathcal T_h$ that belong to $[s_0, a]\times \{0\}$. 

Given $s,\, \sigma\in \mathcal{S}_h$ with $s\leq \sigma$, we first consider a finite element approximation of the auxiliary problem \eqref{eqst}.

Let $\mathbb V_h$ denote the subspace of $H^1(\Omega^+)$ continuous piecewise
bi-linear functions subordinate to $\mathcal T_h(\Omega^+)$, and let $I_h : H^1(\Omega^+) \to \mathbb V_h$ be the Scott-Zhang interpolation \cite{ScottZhang1990} which has the following approximation property
\begin{equation}\label{i:scott-zhang}
  \|v-I_h v\|_{H^1(\Omega^+)}\leq C_{r,\Omega^+}
  h^r\|v\|_{H^{1+r}(\Omega^+)},\quad \text{for } v\in
  H^{1+r}(\Omega^+),
\end{equation}
where $C_{r,\Omega^+}$ depends on $r$, $\Omega^+$, and on the shape regularity of $\mathcal{T}_h$, but is independent of $h$.

Let $\mathbb
V^{\sigma}_0$ be the space of test functions in
$H^1(\Omega^+)$ whose trace vanish on $\Gamma_\sigma^a$ and on
$[0,a]\times \{b\}$, and let $\mathbb
V^\sigma_{h,0} \subset \mathbb V_0$ be its discrete counterpart. Moreover, let  $\mathbb V^\sigma_{w}$ be the  the space of $H^1(\Omega^+)$ functions whose trace on $[0,a]\times
\{b\}$ coincides with the function $w$ and that vanish on $\Gamma_\sigma^a$, and let $\mathbb V^\sigma_{h,w}$ be the space of their Scott-Zhang interpolation onto $\mathbb V^\sigma_h$.

The discrete counterpart of problem \eqref{eqst} then reads: Given the
current boundary condition $w_i$, a crack tip opening  $s\in \mathcal{S}_h$, and $\sigma \in \mathcal{S}_h$,  with $s\leq \sigma$, find $u^{s,\sigma}_{i,h} \in \mathbb V^\sigma_{h,w_i}$ satisfying
\begin{equation}\label{eqst-discrete}
	\int_{\Omega^+}\nabla u^{s,\sigma}_{i,h}\nabla v_h \d x \d y = \int_{\Gamma_s^\sigma}
        \frac{\beta}{\alpha} 
        v_h \, \d x,\qquad\forall v_h\in \mathbb V^\sigma_{h,0} .
\end{equation}
Since the subdivision $\mathcal T_h(\Omega^+)$ is conforming w.r.t. both
$s$ and $\sigma$, i.e., both $s$ and $\sigma$ coincide with
vertices of $\mathcal T_h(\Omega^+)$, then each term in
\eqref{eqst-discrete} can be computed exactly, resulting in a
symmetric and positive definite linear system of equations whose
solution can be computed numerically using a standard preconditioned
conjugate gradient method. By~\cite[Lemma 5.4 and Lemma 5.5]{DM-Toa2},
there exists $r\in(0,1)$ such that $u^{s,\sigma}_i \in H^{1+r}(\Om^+)$,
and we conclude that the numerical solution converges to the exact solution with the following estimate:
\begin{equation}
  \label{eq:fem-estimates}
  \| u^{s,\sigma}_i - u^{s,\sigma}_{i,h} \|_{H^1(\Om^+)} \leq
  C_{r,\Om^+} h^r \| u^{s,\sigma}_i \|_{H^{1+r}(\Om^+)}. 
\end{equation}

An approximate solution to problem~\eqref{minii*} is then obtained by induction on $i$ through the solution of a series of auxiliary problems~\eqref{eqst-discrete}. To be precise, suppose that we know an approximate solution $(u_{i-1,h}, s_{i-1, h})$ of~\eqref{minii*} at the time step $i-1$.  To construct $(u_{i,h}, s_{i,h})$, we proceed as follows.

For any $s\in \mathcal{S}_h$ with $s\geq s_{i-1,h}$, we consider $\sigma^s_{i,h}$ the largest $\sigma \in \mathcal{S}_h$ such that $\sigma \geq s$ and $u^{s, \sigma}_{i,h} \geq 0$.  Thanks to~\eqref{monotonicity of sigma}, it is enough to consider only $\sigma \geq s \lor \sigma^{s_{i-1,h}}_{i-1,h}$.

We set $u^s_{i,h} = u^{s,\sigma}_{i,h}$ with $\sigma=\sigma^s_{i,h}$. According to Lemma~\ref{main lemma**}, $u^s_{i,h}$ can be considered as the discrete counterpart of the solution $u^s_i$ of problem~\eqref{min given s}.

The final step of our algorithm consists in finding a solution $s_{i,h} \in \mathcal{S}_h$ to the minimum problem
\begin{equation*}
\label{discrete energy}
\min_{\substack{s \in \mathcal{S}_h \\ s\geq s_{i-1,h}}
}\Big\{ 
\frac \alpha 2 \int_{\Om^+}|\nabla u^s_{i,h}|^2\d x\d y+ 
\beta \int_{\Gamma_{s}^a}|u^s_{i,h}-u_{i-1,h}|\d x+\frac\gamma 2(s-s_{i-1, h})
\Big\}.
\end{equation*}
The corresponding discrete solution at time index $i$ is given by $u_{i,h}:=u^{s_{i,h}}_{i,h}$, so that the pair $(u_{i,h}, s_{i,h})$   can be considered as the discrete counterpart of the solution of problem~\eqref{minii*}.

\section{A numerical example}

The example in this section has been implemented using the
\texttt{deal.II} library~\cite{dealii9.1, dealii2020, deal2lkit}. 
We consider the domain $\Omega^+ = [0,a]\times[0,b]$, with $a=2$ and
$b=0.5$, set the initial crack tip to $s_0=0.1$, and evolve the
problem to a final time $T=2.5$ for fixed values of $\alpha = 100$ and
$\gamma = 0.5$, varying $\beta$ in the set $\{20, 40, 80, 160\}$. We
consider the time dependent boundary condition
\begin{equation}
  \label{eq:experiments-boundary-condition}
  w_i(x) =
  c_1\left(\frac12-\frac{1}{\pi}\arctan\left(\frac{(x-t_i-s_0)c_2\pi}{c_1}\right)\right),
  \qquad i\geq 0\,,
\end{equation}
for $c_1=0.1$ and $c_2=0.2$, which is shown in
Figure~\ref{fig:boundary-condition} for some time steps, and set $u_0 = 0$.

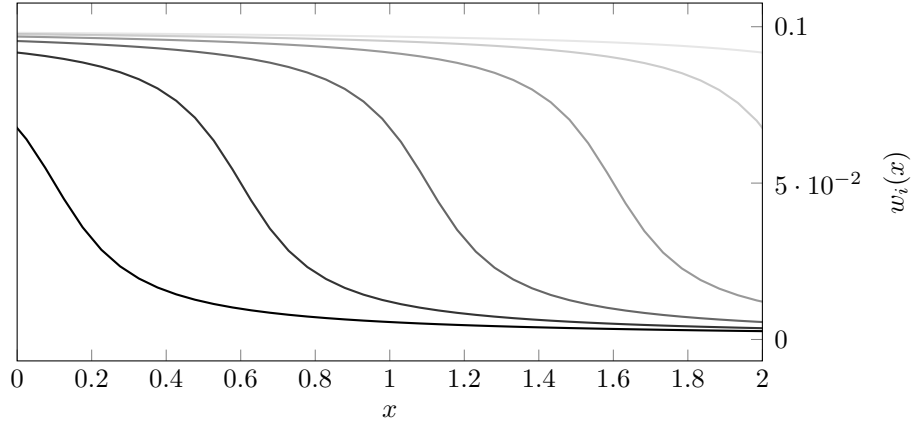
\begin{figure}
  \centering
  \tikzsetnextfilename{boundary_conditions}
  \begin{tikzpicture}
    \begin{axis}[
      xmin=0,
      xmax=2,
      samples=200,
      width=.9\textwidth,
      height=.5\textwidth,
      ytick pos=right,
      xlabel=$x$,
      ylabel=$w_i(x)$
      ]
      \addplot[thick, opacity=1.0]{.1*(1.0/2.0-rad(atan((x-0.0-0.1)*.2*3.14159265359/.1))/3.14159265359)};
      \addplot[thick, opacity=0.8]{.1*(1.0/2.0-rad(atan((x-0.5-0.1)*.2*3.14159265359/.1))/3.14159265359)};
      \addplot[thick, opacity=0.6]{.1*(1.0/2.0-rad(atan((x-1.0-0.1)*.2*3.14159265359/.1))/3.14159265359)};
      \addplot[thick, opacity=0.4]{.1*(1.0/2.0-rad(atan((x-1.5-0.1)*.2*3.14159265359/.1))/3.14159265359)};      
      \addplot[thick, opacity=0.2]{.1*(1.0/2.0-rad(atan((x-2.0-0.1)*.2*3.14159265359/.1))/3.14159265359)};            
      \addplot[thick, opacity=0.1]{.1*(1.0/2.0-rad(atan((x-2.5-0.1)*.2*3.14159265359/.1))/3.14159265359)};            
    \end{axis}
  \end{tikzpicture}
  \caption{Evolution of the boundary condition $w_i$ for $t_i = 0,
    0.5, 1.0, 1.5, 2.0, 2.5$, from left to right.}
  \label{fig:boundary-condition}
\end{figure}

\begin{figure}
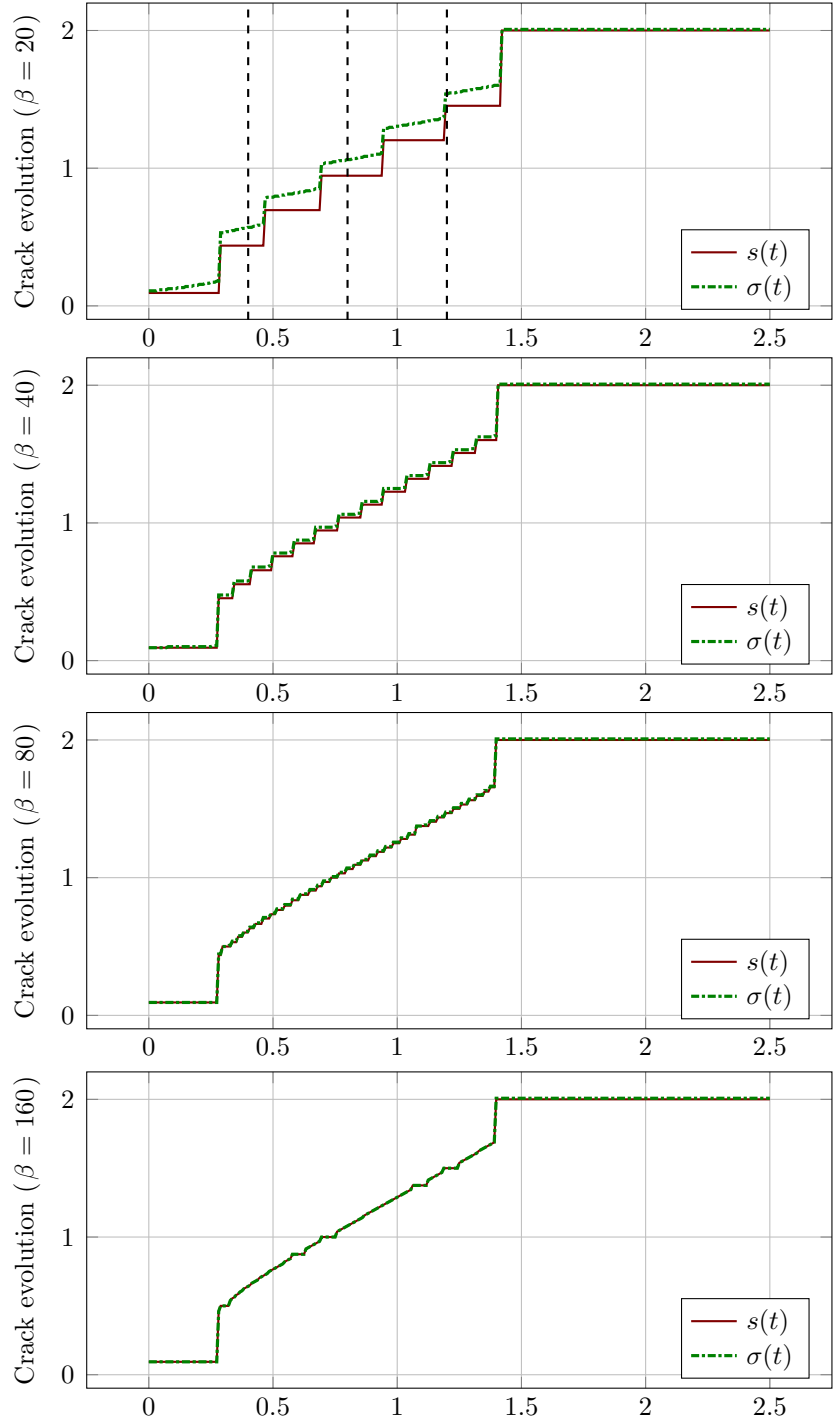

  \centering
  \FrontsPlot{beta_0020_fronts.txt}{20}{%
    \draw [thick, dashed] (.4,-0.1) -- (.4,2.2);%
    \draw [thick, dashed] (.8,-0.1) -- (.8,2.2);%
    \draw [thick, dashed] (1.2,-0.1) -- (1.2,2.2);%
  }
  
  \FrontsPlot{beta_0040_fronts.txt}{40}{}

  \FrontsPlot{beta_0080_fronts.txt}{80}{}

  \FrontsPlot{beta_0160_fronts.txt}{160}{}
  
  \caption{Crack fronts propagation for $\beta = 20, 40, 80$, and
    $160$. The vertical dash lines in the plot for $\beta=20$ show the
    times at which the solution is plotted in
    Figures~\ref{fig:solution-20} and
    \ref{fig:solution-antiplane-20}.}
  \label{fig:crack-fronts}
\end{figure}

Figure~\ref{fig:crack-fronts} shows clearly the jerky evolution of the
crack tip when $\beta$ is small. As the value of $\beta$ increases, the crack tip evolves with more and more jumps, and the cohesive zone becomes smaller and
smaller. When $\beta$ is very large, the evolution is essentially
brittle, and the cohesive zone cannot be detected by the numerical algorithm. This is coherent with the results presented in~\cite[Subsection 8.2]{DM-Toa}.

\begin{figure}
  \centering

  \SolutionPlot{0}{0.4}{
      \legend{
        $u_{h,i}(x,0)$ \\
        $u_{h,i}(x,b) = w_i(x)$ \\
        Cohesive zone \\
      };    
    }

  \SolutionPlot{1}{0.8}{}

  \SolutionPlot{2}{1.2}{}
  
  \caption{Solution at $y=0$ and $y=0.5$, for $t=0.4, 0.8$, and $1.2$
    for the case $\beta=20$. The lower part is added by simmetry.}
  \label{fig:solution-20}
\end{figure}

\begin{figure}
  \centering

   \includegraphics[height=.33\textheight]{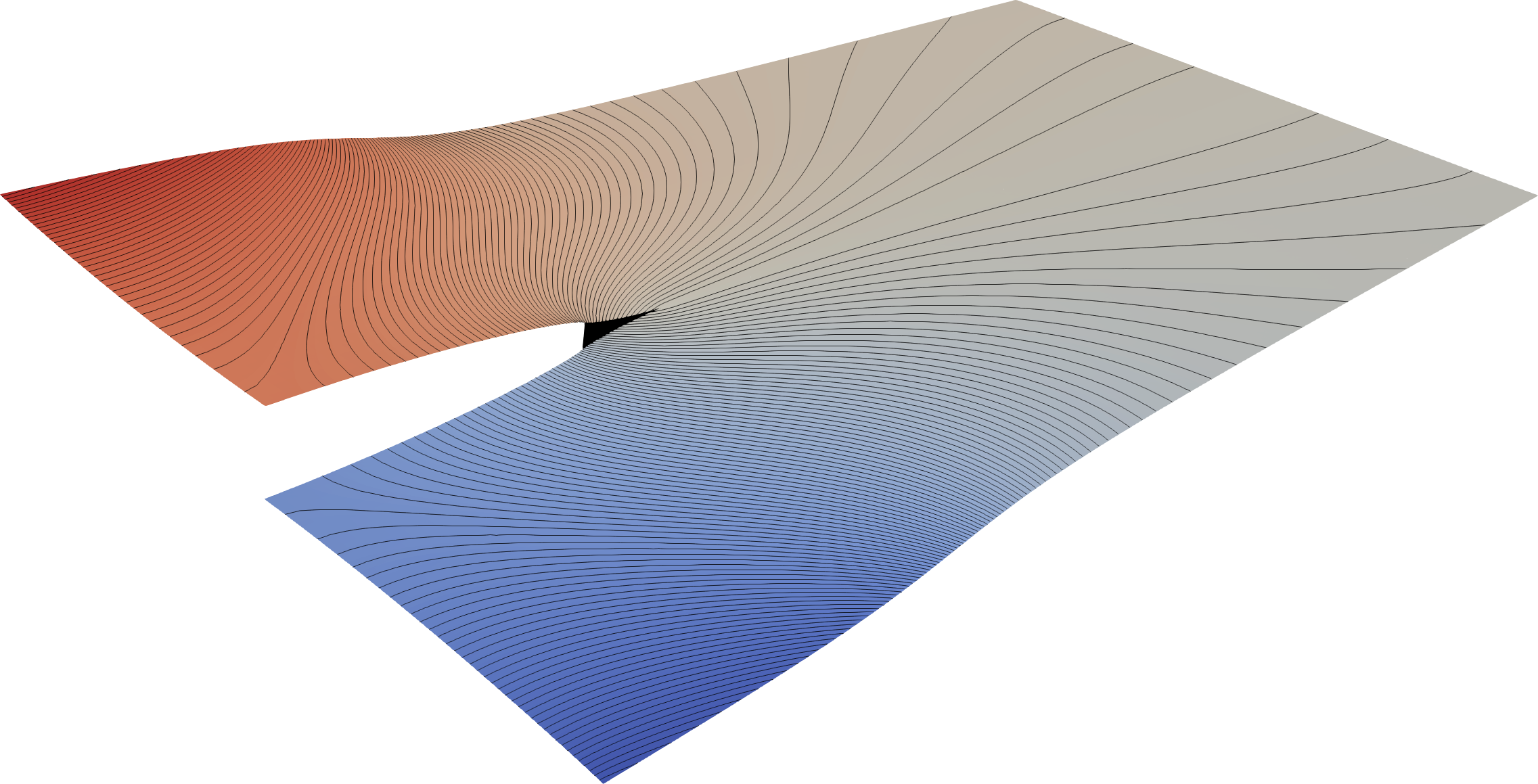}

   \includegraphics[height=.33\textheight]{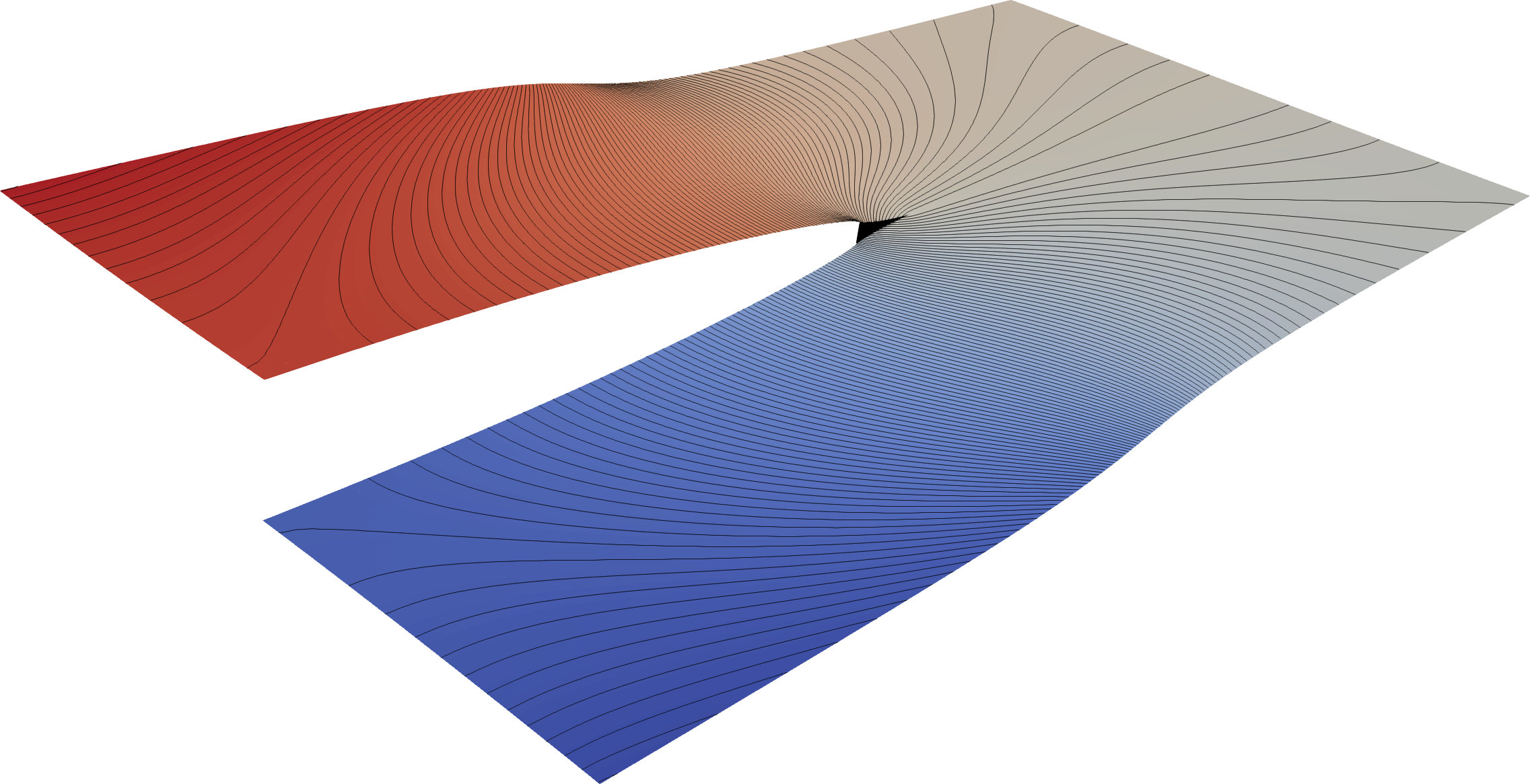}

  \includegraphics[height=.33\textheight]{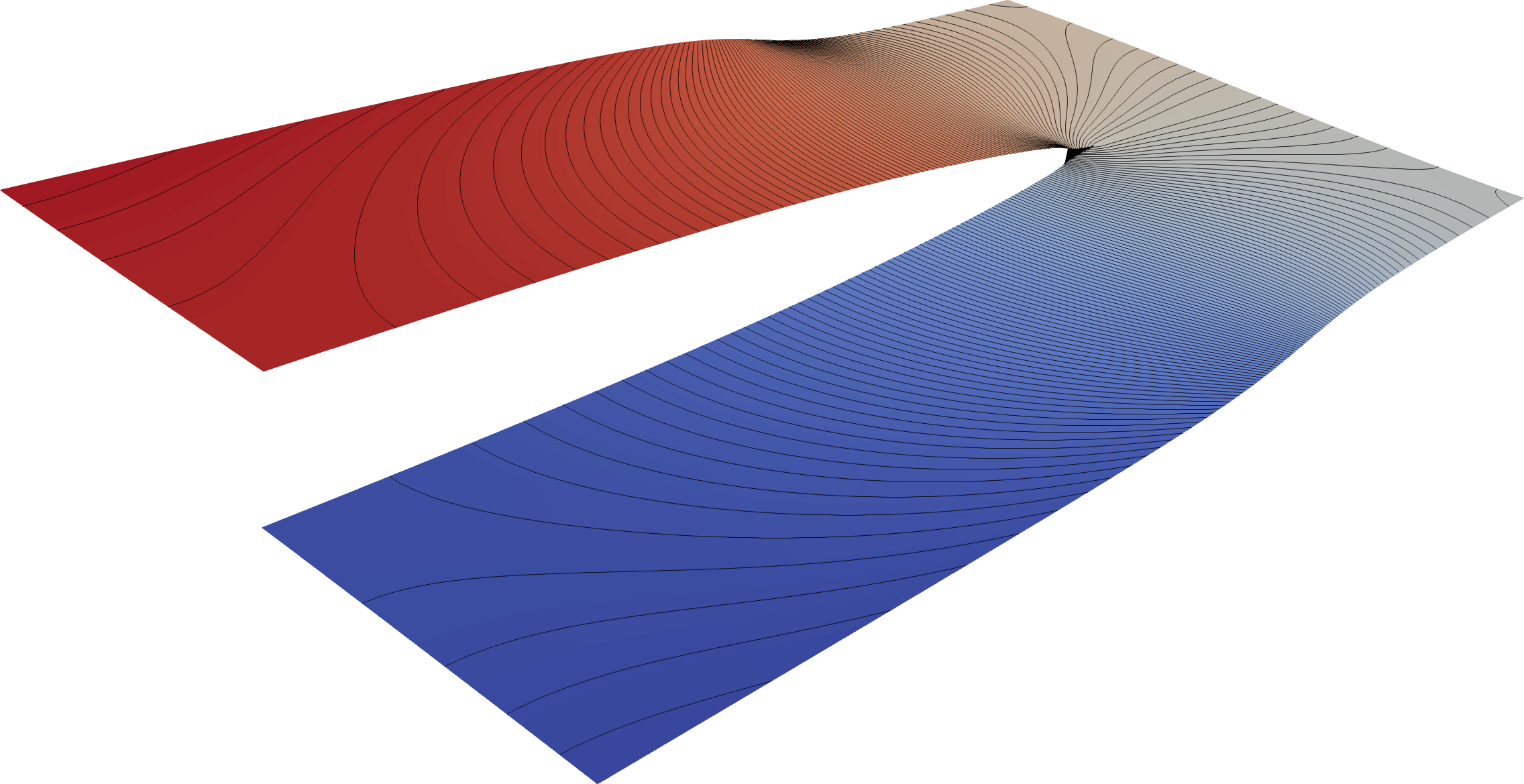}

  \caption{Antiplane representation of the solution at $t=0.4, 0.8$,
    and $1.2$ for the case $\beta=20$.}
  \label{fig:solution-antiplane-20}
\end{figure}

\vspace{1cm}

\noindent \textsc{Acknowledgements.}
 This paper is based on work supported by the National Research Projects (PRIN  2017) ``Variational methods for stationary and evolution problems with singularities and interfaces'', and ``Numerical Analysis for Full and Reduced Order Methods for the efficient and accurate solution of complex systems governed by Partial Differential Equations'', funded by the Italian Ministry of Education, University, and Research.

 \vspace{1cm}

\bibliographystyle{abbrv}
\bibliography{references}

\end{document}